\documentclass[12pt]{amsart}

\usepackage{mathrsfs}
\usepackage{amsfonts}
\usepackage{amssymb}
\usepackage{amsfonts, amscd, amsmath, mathrsfs, amssymb, amsthm, amsxtra, bbding, epsfig, graphicx, latexsym, url}
\usepackage[papersize={8.2in,11in},textwidth=16.2cm,textheight=22cm,centering]{geometry}
\usepackage{enumerate}
\usepackage{caption} 

\usepackage{mathtools}
\usepackage{enumitem}
\usepackage{adjustbox}
\usepackage{microtype}

\usepackage{graphicx}
\usepackage{tikz}
\usepackage{float}

\usepackage{xcolor}
\definecolor{cite}{rgb}{0.00,0.60,1.00}
\definecolor{url}{rgb}{1.00,0.10,0.80}
\definecolor{link}{rgb}{0.00,0.00,1.00}
\usepackage[colorlinks,linkcolor=link,urlcolor=url,citecolor=cite,pagebackref,breaklinks]{hyperref}

\hypersetup{
	pdfstartpage=1,
	pdfstartview=FitH}

\DeclareFontFamily{U}{mathx}{\hyphenchar\font45}
\DeclareFontShape{U}{mathx}{m}{n}{
	<5> <6> <7> <8> <9> <10>
	<10.95> <12> <14.4> <17.28> <20.74> <24.88>
	mathx10
}{}
\DeclareSymbolFont{mathx}{U}{mathx}{m}{n}
\DeclareMathAccent{\widecheck}{\mathalpha}{mathx}{"71}

\numberwithin{equation}{section}

\allowdisplaybreaks

\newtheorem{theorem}{Theorem}[section]
\newtheorem{lemma}{Lemma}[section]
\newtheorem{problem}{Problem}[section]
\newtheorem{definition}{Definition}[section]

\makeatletter
\newcounter{roem}
\renewcommand{\theroem}{\Roman{roem}}

\newcommand{\c@org@eq}{}
\let\c@org@eq\c@equation
\newcommand{\org@theeq}{}
\let\org@theeq\theequation

\newcommand{\setroem}{
	\let\c@equation\c@roem
	\let\theequation\theroem}

\newcommand{\setarab}{
	\let\c@equation\c@org@eq
	\let\theequation\org@theeq}
\makeatother

\newtheorem{remark}{\bf Remark}

\DeclareMathOperator{\Mod}{mod}

\renewcommand{\bmod}[1]{\,(\Mod{ #1})}

\newcommand{\F}{\mathbb{F}}
\newcommand{\N}{\mathbb{N}}

\newcommand{\bb}{\mathbf{b}}

\newcommand{\cA}{\mathcal{A}}

\newcommand{\cC}{\mathcal{C}}

\newcommand{\cE}{\mathcal{E}}
\newcommand{\cF}{\mathcal{F}}

\newcommand{\cM}{\mathcal{M}}

\newcommand{\cP}{\mathcal{P}}

\newcommand{\cS}{\mathcal{S}}

\newcommand{\cW}{\mathcal{W}}

\def\le{\leqslant}
\def\leq{\leqslant}
\def\ge{\geqslant}
\def\geq{\geqslant}

\usepackage{graphicx}
\usepackage{tikz}

\title{An asymptotic Sidon basis of order $3-\eta$}
\author{Wei Niu}
\address{(Wei Niu) School of Mathematics and Statistics, Xi’an Jiaotong University, Xi’an 710049, P. R. China.}
\email{wei.niu@stu.xjtu.edu.cn}

\begin{document}
	
	\begin{abstract}
		\normalsize
		Pilatte recently proved that there exists an infinite Sidon set of positive integers which is an asymptotic basis of order $3$, answering a problem posed by Erd\H{o}s, S\'ark\"ozy and S\'{o}s in 1994.
		In this paper, we strengthen this result by proving that for any $0<\eta<0.0527$, there exists an infinite Sidon set $\mathcal{S}\subset \mathbb{N}$ which is an asymptotic basis of order $3-\eta$; that is, every sufficiently large integer $m$ can be represented as
		\[
		m=s_1+s_2+s_3
		\]
        for some $s_1,s_2,s_3\in \mathcal{S}$ satisfying
		\[
		\min\{s_1,s_2,s_3\}\leq m^{1-\eta}.
		\]
		To prove this, we develop a truncated version of Pilatte's construction and use a deep result of Sawin on sums of Dirichlet convolutions of the von Mangoldt function over function fields.
	\end{abstract}

\subjclass[2020]{11B13, 11R58, 11N13, 05D40}
\keywords{Sidon sets, asymptotic bases, probabilistic method,
additive combinatorics, function fields}

	\maketitle

	\section{Introduction}
	\subsection{Sidon sets}
	A subset $\cS$ of the natural numbers is called a Sidon set if all sums $s_1+s_2$, with $s_1,s_2\in \cS$ and $s_1\le s_2$, are distinct. Introduced by Sidon \cite{Sid32} in 1932, Sidon sets are now widely regarded as one of the central objects in additive combinatorics. For historical background, one can refer to the excellent book of Halberstam and Roth \cite{HR83}, the survey of O'Bryant \cite{OBr04}, and also the book of Guy \cite[C9]{Gu04}.
	
	A fundamental question is to determine how large a Sidon set can be. For  the finite case, let $f(n)$ be the largest cardinality of a Sidon set taken from $\{1,2,\ldots,n\}$.
	It is well known that 
	\begin{align}\label{eq:asy:f}
	f(n)\sim n^{1/2},\ n\to+\infty.
	\end{align}
The upper bound is due to Erd\H{o}s and Tur\'an \cite{ET41}, while the lower bound follows from the  constructions of Singer
\cite{Sin38} and Bose \cite{Bo42}. The asymptotic formula \eqref{eq:asy:f} was recorded
by Chowla \cite{Ch44} and by Erd\H{o}s \cite{Er44}. Moreover,  Erd\H{o}s \cite{Er94} conjectured that for any $\varepsilon>0$,
	\begin{align*}
		f(n)=n^{1/2}+O_{\varepsilon}(n^\varepsilon).
	\end{align*}
	The lower bound of $f(n)$ is related to the existence of primes in short intervals, see, for example, \cite[Section~4.1, p.~9; see also p.~18]{OBr04}. For the upper bound of $f(n)$,  an explicit form was later given by Lindstr\"om \cite{Li69}, and then by Ruzsa \cite{Ru93}, Graham \cite{Gr96} and Cilleruelo \cite{Ci10}. Recently, Balogh, F\"uredi and Roy \cite{BFR23} and O'Bryant \cite{OBr25} obtained new upper bounds for $f(n)$. See Erd\H{o}s Problem~\#30 in
	Bloom's online database \cite{Bl} of Erd\H{o}s problems for more details.
	
	The infinite case  seems more difficult. For any subset $\cA\subset \N$, let $\cA(x)$ denote the number of elements of $\cA$ not exceeding $x$. The greedy algorithm shows that there is an infinite Sidon set $\cS$ such that $\cS(x)\gg x^{1/3}$.
	Ajtai, Koml\'os and Szemer\'edi \cite{AKS81} proved that there is an infinite Sidon set $\cS$ such that 
	$$
	\cS(x)\ge 10^{-3}(x\log x)^{1/3}
	$$
	for all sufficiently large $x$. The exponent-$1/3$ barrier was overcome by Ruzsa \cite{Ru98}, who proved that there is an infinite Sidon set $\cS$ such that 
	$$
	\cS(x)= x^{\sqrt{2}-1+o(1)},
	$$
	which is still the best-known result. Cilleruelo \cite{Ci14} later gave an alternative proof based on an explicit construction. It is still a major open problem \cite{Er56} to determine whether there is an infinite Sidon set $\cS$ satisfying 
	$$
	\cS(x)\gg x^{1/2-\varepsilon}
	$$
	for all $\varepsilon>0$. See Erd\H{o}s Problem \#39 \cite{Bl} for more details.
	On the other hand, Erd\H{o}s showed that $\cS(x)$ could be very close to $\sqrt{x}$ for infinitely many $x$. More precisely, he proved that
    there exists an infinite Sidon set $\cS$ such that
	$$
\limsup_{x\rightarrow+\infty}\frac{\cS(x)}{\sqrt{x}}\ge \frac{1}{2}.
	$$
	He also proved that \begin{align}\label{equation-1}
\liminf_{x\rightarrow+\infty}\frac{\cS(x)\sqrt{\log x}}{\sqrt{x}}<+\infty
	\end{align}
	for any infinite Sidon set $\cS$.
    See \cite{St55} for details.
 Kr\"uckeberg \cite{Kr61} later improved the constant $1/2$ to $1/\sqrt{2}$, a bit closer to the constant $1$ as conjectured by Erd\H{o}s \cite{Er80}. More details can be found in Bloom's list of Erd\H{o}s Problem \#329 \cite{Bl}.

\subsection{Asymptotic Sidon bases}
	A set $\cA\subset\N$ is called an \textit{asymptotic basis of order $k\geq 1$} if every sufficiently large integer can be written as the sum of $k$ elements of $\cA$. Clearly, it follows from (\ref{equation-1}) that no Sidon set can be an asymptotic basis of order $2$. 
	Erd\H{o}s, S\'ark\"ozy and S\'{o}s asked the following problem in \cite[Problem~14]{ESS94a}; see also \cite[Problem~8]{ESS94b}:
	\begin{problem}[Erd\H{o}s--S\'ark\"ozy--S\'{o}s, 1994]\label{problem-1}
		Does there exist an infinite Sidon set which is also an asymptotic basis of order $3?$
	\end{problem}
	The same problem was asked repeatedly by Erd\H{o}s \cite{Er94} and S\'ark\"ozy \cite[Problem 32]{Sar01}. 
	Deshouillers and Plagne \cite{DP09} gave an affirmative answer to Problem \ref{problem-1}  with the order $7$ in place of $3$. The constant $7$ was refined by Kiss \cite{Ki10} to $5$, by Kiss, Rozgonyi and S\'andor \cite{KRS14} to $4$, and by Cilleruelo \cite{Ci15} to $3+\varepsilon$ for any $\varepsilon>0$. Here, an infinite Sidon set $\cS$ being an asymptotic basis of order $3+\varepsilon$ means that every sufficiently large integer $n$ has a representation
	$$
	n=s_1+s_2+s_3+s_4, \quad s_i\in \cS \quad( 1\le i\le 4),
	$$
	where the smallest summand above is at most $n^\varepsilon$. Later, Kiss and S\'andor \cite{KS23} proved that there exists an infinite Sidon set $\cS$ such that the $3$-fold sumset $\cS+\cS+\cS$ contains a positive proportion of natural numbers.  
    
    Several of these results \cite{Ci15,Ki10,KRS14,KS23} are based on the probabilistic method introduced by Erd\H{o}s
	and R\'enyi \cite{ER60}: construct a random set $\cA$ by selecting each integer $n\in \N$ independently with some probability $p(n)$, and then use standard probabilistic inequalities, together with the Borel--Cantelli lemma to show that $\cA$ satisfies the desired properties with probability one. See the book of Halberstam and Roth \cite{HR83} for more details.

    In a recent breakthrough, Pilatte \cite{Pi24} solved Problem \ref{problem-1} in the affirmative,  starting from Cilleruelo's construction of infinite Sidon sets \cite{Ci14}.
    Pilatte constructs a carefully designed auxiliary set $\cA$ (see Lemma \ref{lem:aux}) by using the alteration method
    \footnote{The alteration method starts with a random structure that has almost all the desired properties, and then removes a small number of bad configurations to obtain the required structure. See \cite[Chapter 3]{AS16} for more details.}. The elements of $\cA$ play
    the same separating role as the fixed number $4$ and the range of $x_j(p)$ in Cilleruelo's
    construction\footnote{The construction can be found at \cite[p.478]{Ci14}.}, while the structure of $\cA$ provides enough flexibility to
    ensure that the resulting set is an asymptotic basis of order $3$.  Pilatte's proof also relies on  a deep result of Sawin \cite{Saw24} on sums of Dirichlet convolutions of von Mangoldt functions over function fields. We use the same result of
	Sawin in a different form, recorded below as Lemma \ref{lem:sawin}.
	
	In fact, Pilatte's arguments yield the following quantitative version: Every sufficiently large integer $m$ can be written as
	$$
	m=s_1+s_2+s_3,
	$$
where
	each summand $s_i$ in the representation  of $m$ satisfies
	$$
	me^{-O((\log m)^{1/2})}\leq s_i\le m.
	$$ See \cite[Lemma 3.4 (ii)]{Pi24} and \cite[Proof of Lemma 5.3]{Pi24}.
	Motivated by this observation and by the $3+\varepsilon$ result of Cilleruelo \cite{Ci15}, Ping Xi asked the author the following $3-\eta$ problem:
    
	\begin{problem}[$3-\eta$]\label{problem-2}
		Find an admissible constant $\eta>0$, as large as possible, such that there exists an infinite Sidon set $\cS\subset \N$ which is an asymptotic Sidon basis of order $3-\eta$, that is, all sufficiently large integers $m$ can be represented as
		\[
		m=s_1+s_2+s_3,\quad s_1,s_2,s_3\in \cS,
		\]
		where
		\[
		\min\{s_1,s_2,s_3\}\leq m^{1-\eta}.
		\]
	\end{problem}
We give a positive answer for a small but explicit range of $\eta$.
	
	\begin{theorem}\label{thm:main}
	For any $0<\eta<0.0527,$ there exists an infinite Sidon set $\cS\subset \N$ which is an asymptotic Sidon basis of order $3-\eta$.
	\end{theorem}

\subsection{Structure of this paper}
	
	The rest of our article is organized as follows. In Section \ref{Sec-2}, we fix some notation and state some preliminary results used later. In Section \ref{Sec-3}, we provide the construction of our desired Sidon set following the arguments of Pilatte \cite{Pi24}, but we need to develop a truncated version. In Section \ref{Sec-4}, we prove that the constructed set indeed satisfies the Sidon property. In addition, Lemma \ref{lem:sizeS} and \eqref{eq:parameters} provide a quantitative description of the set $\cS$ in Theorem \ref{thm:main}. In Section \ref{Sec-5}, we show that our constructed Sidon set is an asymptotic basis of order $3-\eta$ with probability one. In the last section, we conclude the paper and provide some open problems for further research.

    We expect the construction of Pilatte \cite{Pi24}, as well as the truncated version introduced in this paper, can find many other applications in the study of Sidon sets.

	\section{Notation and Preliminaries}\label{Sec-2}
	Let $\mathbb Z$ and $\mathbb{N}$ be the sets of integers and natural numbers, respectively. 
	For $k\ge 1$ and subsets $A_1,\ldots,A_k\subset \mathbb Z$, let 
	\[
	A_1+\cdots+A_k
	:=
	\{a_1+\cdots+a_k: a_i\in A_i \text{ for } 1\le i\le k\}
	\]
	be their sumset.
	Throughout the paper, for real-valued functions $f(x), g(x)$, we write $f(x) = O(g(x))$ or $f(x)\ll g(x)$ if $|f(x)| \leq C|g(x)|$ for some
	constant $C > 0$ and $f(x)=o(g(x))$ with $g(x)\neq 0$ if
$
\lim_{x\to\infty}f(x)/g(x)=0
$. We write $f(x)\asymp g(x)$ if both $f(x)\ll g(x)$ and
$g(x)\ll f(x)$ hold.
	For integers $n$ and $m\geq 1$, the symbol $[n~{\rm mod}~m]$ is defined to be the unique integer $a$ such that $a\in \{0,1,\ldots,m-1\}$ and $a\equiv n\bmod m$.
	
	Following Cilleruelo \cite{Ci14} and Pilatte \cite{Pi24}, we use the following generalized base. 
	\begin{definition}
		Let $\bb=(b_1,b_2,\cdots)$ be an infinite sequence of integers with $b_i\geq2$. For any $m\in \N$ and $a_1,\ldots,a_m\in \mathbb{Z}$, write
		\[
		\overline{a_m\cdots a_1}^{\,\bb}
		:=a_1+a_2b_1+a_3b_1b_2+\cdots+a_m b_1\cdots b_{m-1}.
		\]
		In particular, every positive integer $n$ has a unique representation of this
		form with $0\leq a_i<b_i~(1\le i\le m)$ and $a_m\neq 0$.
	\end{definition}

	Let $q\ge 2$ be a prime power and $\F_q$ the finite field with $q$ elements. Let $\F_q[t]$ be the polynomial ring over $\F_q$.
	Throughout this paper, let $\cM$ and $\cM_n$ denote the sets of monic polynomials and monic polynomials of degree $n$ in $\mathbb{F}_q[t]$, respectively.
	Denote by $\cP_n$ the set of irreducible polynomials in $\cM_n$. 
	
	The following lemma is a standard result, and
	see \cite[Theorem 2.2]{Ro02} for instance.
	
	\begin{lemma}\label{lem:prime}
		For any $n\geq1$,
		\[ |\cP_n|=\frac{q^n}{n}+O\Big(\frac{q^{n/2}}{n}\Big), \] where the implied constant is absolute.
	\end{lemma}
	
	Our next lemma is a consequence of Sawin's result \cite[Lemma 9.14]{Saw24}. 
	\begin{lemma}\label{lem:sawin}
		For $0<\theta<1$ and integers $1\leq d_0< d$, let $n=d_0+2d$ and $g\in\F_q[t]$ be squarefree of degree $2\leq \deg(g)\leq \theta n$, and suppose that no irreducible factor of $g$ has degree $d_0$ or $d$. 
		Then, for any $a\in\big(\F_q[t]/(g)\big)^\times$, the number $T(a;g;d_0,d)$ of triples
		\[
		(f_0,f_1,f_2)\in\cP_{d_0}\times\cP_d\times\cP_d
		\]
		satisfying
		\[
		f_1\ne f_2
		\quad \text{and} \quad
		f_0f_1f_2\equiv a\bmod g
		\]
		is
		\[
		T(a;g;d_0,d)=
		\frac{|\cP_{d_0}|\,|\cP_d|\big(|\cP_d|-1\big)}{\phi(g)}
		+O\!\left(C_\theta^{\,n}q^{\frac{n-\deg(g)}{2}}\right),
		\]
		where $\phi(g)=\big|\big(\F_q[t]/ (g)\big)^\times\big|$ and 
        \[
        C_\theta=\frac{3e(1+2\theta)2^{\frac{\theta}{1-\theta}}}{1-\theta}.
        \]
	\end{lemma}
	
	\begin{proof}
    Since no irreducible factor of $g$ has degree $d_0$ or $d$ and all $f_i$ are irreducible, we have $(f_0f_1f_2,g)=1$. 
		Applying Sawin's lemma \cite[Lemma 9.14]{Saw24} with $h=1$, $c=0$, $\omega=3$, and degree tuple $(d_0,d,d)$, we get from  \cite[(9.17)]{Saw24} that\footnote{The definition of $H^h_{n_1,\ldots,n_\omega}(f)$ in \eqref{eq-sawin-1} is given at the bottom of \cite[p.394]{Saw24}.}
		\begin{align}\label{eq-sawin-1}
			\Bigg|d_0d^2T(a;g;d_0,d)-\frac{1}{\phi(g)}
			\sum_{\substack{f\in\cM_{n}}}
			H^1_{d_0,d,d}(f)\Bigg|\ll\frac{n!}{d_0!(d!)^2}c_\theta^{n-\deg(g)}q^{\frac{n-\deg(g)}{2}}, 
		\end{align}
		where the constant 
        \[
c_{\theta}=\frac{e(1+2\theta)2^{\frac{\theta}{1-\theta}}}{1-\theta}.
        \] 
        By the definition of $H^1_{d_0,d,d}(f)$, since $d_0\neq d$, we have
		\begin{align}\label{eq-sawin-2}
			\frac{1}{\phi(g)}
			\sum_{\substack{f\in\cM_{n}}}
			H^1_{d_0,d,d}(f)=\frac{1}{\phi(g)}
			\sum_{\substack{f\in\cM_{n}}}
			\sum_{\substack{
					(f_0,f_1,f_2)\in(\cM)^3\\
					f_0\in \cP_{d_0},\ f_1,f_2\in \cP_{d}\\
					f_1\neq f_2,\
					f_0f_1f_2=f}}
			d_0d^2 .
		\end{align}
		Therefore, the right-hand side of (\ref{eq-sawin-2}) equals 
		\begin{align}\label{eq-sawin-3}
			\frac{d_0d^2}{\phi(g)}
			\sum_{\substack{f\in\cM_{n}}}
			\sum_{\substack{
					(f_0,f_1,f_2)\in(\cM)^3\\
					f_0\in \cP_{d_0},\ f_1,f_2\in \cP_{d}\\
					f_1\neq f_2,\
					f_0f_1f_2=f}}
			1 =\frac{d_0d^2}{\phi(g)}|\cP_{d_0}|\,|\cP_d|\big(|\cP_d|-1\big).
		\end{align}
		By the multinomial theorem,
		$$
		3^n=(1+1+1)^n
		=
		\sum_{\substack{a+b+c=n\\a,b,c\ge0}}
		\frac{n!}{a!b!c!},
		$$
		which implies that
		$$
		\frac{n!}{d_0!(d!)^2}\le 3^n.
		$$
		Therefore, inserting \eqref{eq-sawin-2} and \eqref{eq-sawin-3} into \eqref{eq-sawin-1}, dividing by $d_0d^2$ and taking $C_\theta=3c_\theta$, we obtain the desired formula.
	\end{proof}
	
	The following auxiliary set, introduced by Pilatte \cite[Lemma 3.1]{Pi24}, is a key ingredient in his construction.
	
	\begin{lemma}[{\cite[Lemma~3.1]{Pi24}}]\label{lem:aux}
		For every sufficiently large prime $p$, there exists a set
		\[
		\cA\subset\big\{1,2,\ldots,\lfloor p/2\rfloor-1\big\}
		\]
		such that $\cA+\cA+\cA$ contains $p+2$ consecutive integers and
		\[
		\cA\cap\big(\cA+\cA+\{0,1\}\big)=\emptyset.
		\]
	\end{lemma}
	\begin{remark}
		In fact, Pilatte proved a stronger result:
		$$
		[p/5,7p/5]\cap \mathbb Z \subset \cA+\cA+\cA.
		$$
	\end{remark}

	\section{Construction}\label{Sec-3}
	\subsection{Parameters}
	Choose positive real numbers $c$ and $\alpha$ subject to the following constraints:
	\begin{equation}\label{eq:parameters}
		\begin{aligned}
			&0<c<\alpha^4<\alpha^2<1,\\
			&(\alpha^2-c)(\alpha^4-c)>c,\\
			&\alpha^4>2c,\\ 
			&c(2+\alpha^2)-1>0.
		\end{aligned}
	\end{equation}
	The following ranges are admissible:
	\[
	0.9472873042\ldots<\alpha^2<1
	\]
	and
	\[
	\frac{1}{2+\alpha^2}<c<
	\frac{
		1+\alpha^2+\alpha^4
		-\sqrt{(1+\alpha^2+\alpha^4)^2-4\alpha^6}
	}{2}.
	\]
    
	We take
	\begin{align}\label{choice}
	\alpha^2=0.9473
	\quad \text{and} \quad
	c=0.3393.
	\end{align}
	With these choices, we have
	\[
	\alpha^4=0.89737729>0.6786=2c,
	\quad c(2+\alpha^2)=1.00001889>1.
	\]
	Moreover,
	\[
	(\alpha^2-c)(\alpha^4-c)
	=0.33931099232>0.3393=c.
	\]
	Put
	\begin{equation}\label{eq:beta-gamma}
		\beta:=c(2+\alpha^2)-\alpha^4
		\quad \text{and} \quad
		\gamma:=c(2+\alpha^2)-1.
	\end{equation}
	Then
	\[
	\beta=0.1026416>0,
	\qquad
	\gamma=0.00001889>0.
	\]
	Keep in mind the above inequalities, which would be useful in our subsequent arguments.
	
	\subsection{A truncated construction}
	For the rest of the proof, fix a (sufficiently large) prime $p$ and a set $\cA$ as in Lemma \ref{lem:aux}, and put
	\[
	\cW=\{0,1,\ldots,p-1\}.
	\]
	Let $q$ be a sufficiently large prime power with $q>p$.
	For each $i\ge1$, choose a polynomial $g_i\in\cP_{2i-1}$, and let $\omega_i$ be one of the generators of the cyclic group $\big(\F_q[t]/(g_i)\big)^\times$. Let $K$ be a fixed sufficiently large integer. For all integers $k\ge K$, define
	\[
	\cF_k:=\bigcup_{ck^2\le 2j<c(k+1)^2}\cP_{2j},
	\]
	and put
	\[
	\cF:=\bigcup_{k\ge K}\cF_k.
	\]
	
	For any $k\ge K$, let $k_\alpha=\lfloor\alpha k\rfloor$, then $k_\alpha<k$ since $\alpha<1$. Let $\bb=(b_1,b_2,\ldots,b_i,\ldots)$ be a base with
	\[
	b_i=
	\begin{cases}
		q^{i}-1, & i \ \text{odd},\\
		p, & i \ \text{even}.
	\end{cases}
	\]  
	For any $f\in\cF$, there exists a unique integer $k\geq K$ such that $f\in \cF_k$. We define  
	\[
	n_f:=\overline{s(f)\ R_k(f)\ E_k(f)\cdots\ R_{k_\alpha+1}(f)\ E_{k_\alpha+1}(f) \ r_{k_\alpha}(f)\ e_{k_\alpha}(f)\ \cdots\ r_1(f)\ e_1(f)}^{\,\bb},
	\]
	where
	\begin{enumerate}[label=(\roman*)] 
		\item For $1\leq i\le k_\alpha$, let $e_i(f)$ be the unique integer satisfying 
		\[ 0\le e_i(f)<q^{2i-1}-1 \quad \text{and} \quad \omega_i^{e_i(f)}\equiv f\bmod {g_i}. 
		\] 
		Since the degree of $g_i$ is odd and the degree of $f$ is even, we obtain $(f,g_i)=1$, which implies that such an $e_i(f)$ exists.
		For $k_\alpha< i\le k$, let $E_i(f)$ be a random variable uniformly distributed on 
		\[ 
		\big\{0,1,\ldots,q^{2i-1}-2\big\}. 
		\] 
        
		\item For $1\le i\le k_\alpha$, let $r_i(f)$ be a random variable uniformly distributed on  $\cA$. For $k_\alpha<i\le k$, let $R_i(f)$ be a random variable uniformly distributed on $\cW$. Recall that $\cA$ and $\cW$ are the sets defined at the beginning of this subsection.
        
		\item Let $s(f)$ be a random variable uniformly distributed on 
		$$
		\big\{1,2,\ldots,q^{3k}\big\}.
		$$ 
	\end{enumerate}
	Moreover, all random variables defined above are required to be mutually independent as $f$ and $i$ vary.

	Let $\cS$ be the random set defined by
	\begin{align}\label{def-S}
		\cS:=\{n_f:f\in\cF\}.
	\end{align}  
	In Section \ref{Sec-4}, we show that $\cS$ is an infinite Sidon set and estimate $\cS(x)$.
	In Section \ref{Sec-5}, we prove that $\cS$ is  an asymptotic
	basis of order $3-\eta$ with probability 1. From now on,  $\cS$ will always denote the set defined by \eqref{def-S}.

The main difference between Pilatte's construction and ours comes from
the restriction on the size of one of the summands. In Pilatte's
argument, the corresponding polynomials belong to the same set $\cP_d$. To make one summand $s_i$ smaller, we choose the corresponding polynomial from $\cP_{d_0}$  with some $d_0<d$. This leads to the
$\bb-$expansion of $s_i$ terminating at an earlier position. To preserve the
asymptotic basis property, the digits in the truncated part must be
chosen from a larger set: we replace the
$\cA$-valued digits in Pilatte's construction by $\cW$-valued digits.
It allows us to rewrite $m$ so that its digits belong
to the corresponding sumsets, see Lemma \ref{lem:expansion}.

	\section{The Sidon property}\label{Sec-4}
	We first give a few properties of the set $\cS$ constructed in the last section.
	\begin{lemma}\label{lem:size}
		If $f\in\cF_k$, then
		\[
		n_f=q^{k^2+O(k)}.
		\]
	\end{lemma}
	
	\begin{proof}
		Clearly, we have
		\begin{align}\label{ineq:1}
			n_f\geq \overline{1 \underbrace{00\cdots 00}_{2k\ \text{zeros}}}^{\,\bb}= b_1b_2\cdots b_{2k}=\prod_{j=1}^{k}p\left(q^{2j-1}-1\right)>\prod_{j=1}^{k}q^{2j-1}=q^{k^2}.
		\end{align}
		On the other hand,
		\begin{align*}
			n_f &\le \overline{q^{3k}(p-1)(q^{2k-1}-2)\cdots(p-1)(q^{2k_\alpha+1}-2)\lfloor p/2\rfloor(q^{2k_\alpha-1}-2)\cdots\lfloor p/2\rfloor (q-2)}^{\,\bb}\\
			&\le \overline{q^{3k} \underbrace{00\cdots 00}_{2k\ \text{zeros}}}^{\bb}+\overline{1 \underbrace{00\cdots 00}_{2k\ \text{zeros}}}^{\bb}\\
			&\le \big(q^{3k}+1\big)p^k\prod_{j=1}^{k}q^{2j-1}.
		\end{align*}
		Recall that $p<q$, thus we obtain
		\begin{align}\label{ineq:2}
			n_f
			\leq q^{k^2+4k+1}.
		\end{align}
		The lemma follows from \eqref{ineq:1} and \eqref{ineq:2}.
	\end{proof}
	
	\begin{lemma}\label{lem:injective}
		For any $f,g\in \cF$, if $n_f=n_g$, then $f=g$.
	\end{lemma}
	
	\begin{proof}
		Suppose that $f\in\cF_k$ and $g\in\cF_\ell$. Then, by Lemma \ref{lem:size}, we have  
		$$
		n_f=q^{k^2+O(k)} \quad \text{and} \quad n_g=q^{\ell^2+O(\ell)},
		$$
		which implies that $k=\ell+O(1)$. Let 
		$$
		L:=\min\big(\lfloor\alpha k\rfloor,\lfloor\alpha \ell\rfloor\big).
		$$
		Since $n_f=n_g$, clearly 
		$$
		\left[
		n_f \ {\rm mod}\  \prod_{i=1}^{2L} b_i
		\right]
		=
		\left[
		n_g \ {\rm mod}\ \prod_{i=1}^{2L} b_i
		\right].
		$$ 
		Note that, for $1\leq i\leq L$,
		\[
		0\leq e_i(f),e_i(g)<q^{2i-1}-1, \ \ 0\leq r_i(f),r_i(g)<p.
		\]
		Hence, we have
		\[
		\overline{r_L(f)\ e_L(f)\ \cdots\ r_1(f)\ e_1(f)}^{\,\bb}
		=
		\overline{r_L(g)\ e_L(g)\ \cdots\ r_1(g)\ e_1(g)}^{\,\bb},
		\]
		Moreover, the corresponding digits are equal, which implies 
		\[
		e_i(f)=e_i(g) \quad (1\le i\le L).
		\]
        And by the definition of $e_i$, we obtain
        \[
        f\equiv g \bmod{g_i} \quad (1\le i\le L).
        \]
		Therefore, by the Chinese Remainder Theorem, we have
		\begin{align}\label{eq:congL}
		f\equiv g\Bigl(\operatorname{mod}\ \prod_{i=1}^{L} g_i\Bigr).
		\end{align}
		Since $k=\ell+O(1)$, the modulus $\prod_{i=1}^{L} g_i$ above has degree
		\[
		\sum_{i\le L}(2i-1)=L^2=\alpha^2 k^2+O(k),
		\]
		and
		\[
		\deg(f-g)\leq \max\{\deg f,\deg g\}\le ck^2+O(k).
		\]
		Since $K$ is sufficiently large and $\alpha^2>c$ from (\ref{eq:parameters}), the congruence \eqref{eq:congL} forces $f=g$.
	\end{proof}
	
	The following lemma shows that $\cS$ is indeed an infinite Sidon set.
	\begin{lemma}\label{prop:sidon}
		The set $\cS$ defined in \eqref{def-S} is a Sidon set.
	\end{lemma}
	\begin{proof}
		Let $n_{f_1},n_{f_2},n_{f_3},n_{f_4}$ be four elements of $\cS$ such that
		\begin{equation}\label{eq:sidon-eq}
			n_{f_1}+n_{f_2}=n_{f_3}+n_{f_4},
		\end{equation}
		It suffices to prove 
		$\{n_{f_1},n_{f_2}\}=\{n_{f_3},n_{f_4}\}$. Suppose that $f_j\in\cF_{k_j}$ $(1\le j\le 4)$.
		We may assume, without loss of generality, that $k_1\ge k_2$ and $k_3\ge k_4$.
		Then, by Lemma \ref{lem:size}, we have
		\[
		q^{k_1^2+O(k_1)}= n_{f_1}+n_{f_2}=n_{f_3}+n_{f_4}=q^{k_3^2+O(k_3)},
		\]
		which implies that 
		\begin{align}\label{eq-size-1}
			k_1=k_3+O(1).
		\end{align}
		Put 
		$$
		L=\min \big(\lfloor\alpha k_1\rfloor,\lfloor\alpha k_3\rfloor\big).
		$$
		Exchanging the order of pairs in \eqref{eq:sidon-eq} if necessary, we may further assume $k_2\ge k_4$.
		Reduce the identity $n_{f_1}+n_{f_2}=n_{f_3}+n_{f_4}$ modulo $\prod_{i=1}^{2L}b_i$, and then write the common residue in base $\bb$. Then, we get
		\[
		\left[n_{f_1}+n_{f_2}\ {\rm mod}\  {\prod_{i=1}^{2L}b_i}\right]= \left[n_{f_3}+n_{f_4}\ {\rm mod}\ {\prod_{i=1}^{2L}b_i}\right]:=\overline{y_{L}\ x_{L}\ \cdots\ y_1\ x_1}^{\,\bb},
		\]
		where $0\leq x_i<q^{2i-1}-1$ and $0\leq y_i<p$ for $1\leq i\leq L$.

		We are going to prove
		\begin{equation}\label{eq:k2k4}
			k_4\ge \alpha k_2-9.
		\end{equation}
		Assume the contrary, i.e., 
		$
		\alpha k_2-5>k_4+4.
		$
		Then, we can choose an integer $\ell$ such that
		\[
		k_4+4<\ell \le \alpha k_2-4.
		\]
		We first show that $\ell< L$. In view of  \eqref{ineq:1} and \eqref{ineq:2}, we know  
		\begin{align}\label{eq-Sidon-1}
			q^{k_i^2}\leq n_{f_i}<q^{k_i^2+4k_i+1}. 
		\end{align}
		Then by \eqref{eq-Sidon-1} we have
		$
		n_{f_1}+n_{f_2}> q^{k_2^2}
		$
		and
		$$
		n_{f_3}+n_{f_4}<q^{k_3^2+4k_3+1}+q^{k_4^2+4k_4+1}\le 2q^{k_3^2+4k_3+1}<q^{k_3^2+4k_3+2}.
		$$
		Hence, by the identity $n_{f_1}+n_{f_2}=n_{f_3}+n_{f_4}$ we have
		$$
		k_2^2<k_3^2+4k_3+2<(k_3+2)^2,
		$$
		which means $k_2<k_3+2$. Therefore,
		$$
		\ell\le \alpha k_2-4<\alpha (k_3+2)-4<\alpha k_3-2.
		$$
		Moreover, we also have 
		$
			\ell \le \alpha k_1-4,
		$
		from which it follows that $\ell<L$.
		Since $\ell<L$, we have
		\[
		\left[n_{f_1}+n_{f_2}\ {\rm mod}\ {\prod_{i=1}^{2\ell}b_i}\right]=\left[n_{f_3}+n_{f_4}\ {\rm mod}\ {\prod_{i=1}^{2\ell}b_i}\right]=\overline{y_{\ell}\ x_{\ell}\ \cdots\ y_1\ x_1}^{\,\bb}.
		\]
		For better clarity, we refer to the following table throughout the remainder of the proof.
\begin{table}[htbp]
\renewcommand{\arraystretch}{1.7}
\renewcommand{\tabcolsep}{1.8mm}
\begin{tabular}{c|ccccc}
\hline$n_{f_1}$ & $r_{\ell}(f_1)$ & $e_{\ell}(f_1)$ & $\cdots$ & $r_1(f_1)$ & $e_1(f_1)$ \\
$n_{f_2}$ & $r_{\ell}(f_2)$ & $e_{\ell}(f_2)$ & $\cdots$ & $r_1(f_2)$ & $e_1(f_2)$ \\
\hline$n_{f_1}+n_{f_2}$ & $y_{\ell}$ & $x_{\ell}$ & $\cdots$ & $y_1$ & $x_1$ \\
\hline
\end{tabular}
\end{table}

		For $1\leq i\leq \ell$, we have 
		$$
		x_i=\Big[e_{i}(f_1)+e_{i}(f_2)\ {\rm mod}\ {q^{2i-1}-1}\Big]<q^{2i-1}-1
		$$ and
		$$
		y_i= r_{i}(f_1)+r_{i}(f_2)+\mathbf{1}_{e_{i}(f_1)+e_{i}(f_2)\geq q^{2i-1}-1}\le 2(\lfloor p/2\rfloor-1)+1<p.
		$$
		
		For the summation $n_{f_3}+n_{f_4}$, we refer to the following table.
\begin{table}[htbp]
\renewcommand{\arraystretch}{1.7}
\renewcommand{\tabcolsep}{1.8mm}
\begin{tabular}{c|cccccccc}
\hline 
$n_{f_3}$ & $r_{\ell}(f_3)$ & $e_{\ell}(f_3)$ & $\cdots$ & $r_{k_4+1}(f_3)$ & $e_{k_4+1}(f_3)$ & $r_{k_4}(f_3)$ & $e_{k_4}(f_3)$ & $\cdots$\\
$n_{f_4}$ & 0 & 0 & $\cdots$ & 0 & $s(f_4)$ & $R_{k_4}(f_4)$ & $E_{k_4}(f_4)$ & $\cdots$ \\
\hline 
$n_{f_3}+n_{f_4}$ & $y_{\ell}$ & $x_{\ell}$ & $\cdots$ & $y_{k_4+1}$ & $x_{k_4+1}$ & $y_{k_4}$ & $x_{k_4}$ & $\cdots$ \\
\hline
\end{tabular}
\end{table}

		For $1\leq i\leq \alpha k_4$, we have
		$$
		x_i=\Big[e_{i}(f_3)+e_{i}(f_4)\ {\rm mod}\  {q^{2i-1}-1}\Big]<q^{2i-1}-1
		$$ 
		and
		\[
		y_i= r_{i}(f_3)+r_{i}(f_4)+\mathbf{1}_{e_{i}(f_3)+e_{i}(f_4)\geq q^{2i-1}-1}\le 2(\lfloor p/2\rfloor-1)+1<p.
		\]
		Since 
		$$
		n_{f_4}<q^{(k_4+2)^2}<\overline{1 \underbrace{00\cdots 00}_{2k_4+4\ \text{zeros}}}^{\,\bb}
		$$ 
		from \eqref{ineq:2}, $\ell>k_4+4$ and $r_{\ell-1}(f_3)+1\leq p/2$ , we see that $n_{f_4}$ does not contribute either a separator digit or a carry to $y_\ell$. Thus, we obtain
		$$
		y_\ell=r_{\ell}(f_3)\in \cA.
		$$
		While,
		\[
		y_\ell= r_{\ell}(f_1)+r_{\ell}(f_2)+\mathbf{1}_{e_{\ell}(f_1)+e_{\ell}(f_2)\geq q^{2\ell-1}-1}\in \cA+\cA+\{0,1\}.
		\]
		This contradicts the property of $\cA$ in Lemma \ref{lem:aux}:  
		\[
		\cA\cap(\cA+\cA+\{0,1\})=\emptyset,
		\]
		giving the proof of \eqref{eq:k2k4}.
		
		\smallskip
		Based on the above analysis, for $1\leq i\leq \alpha k_4$, we obtain
		\[
		x_i=\Big[e_{i}(f_1)+e_{i}(f_2)\ {\rm mod}\ {q^{2i-1}-1}\Big]=\Big[e_{i}(f_3)+e_{i}(f_4)\ {\rm mod}\ {q^{2i-1}-1}\Big],
		\]
		and for $ k_2+4<i\le L$, we obtain
		\[
		x_i=e_{i}(f_1)=e_{i}(f_3).
		\]
		This implies, by the definition of $e_i$ and the Chinese Remainder theorem, that
		\begin{equation}\label{eq:prod-cong}
			f_1f_2\equiv f_3f_4\bmod {G},
			\qquad
			G:=\prod_{i\le \alpha k_4}g_i,
		\end{equation}
		and
		\begin{equation}\label{eq:single-cong}
			f_1\equiv f_3\bmod H,
			\qquad
			H:=\prod_{k_2+4<i\le  L}g_i.
		\end{equation}
		Now, we divide the discussion into three cases.
		
		{\it Case I.} $f_1=f_3$. In this case, we have $n_{f_1}=n_{f_3}$, and hence $n_{f_2}=n_{f_4}$ from \eqref{eq:sidon-eq}.
		
		{\it Case II.} $f_1f_2=f_3f_4$. Since $f_i$ $(1\leq i\leq 4)$ is irreducible in $\F_q[t]$, we have
		\[
		\{f_1,f_2\}=\{f_3,f_4\},
		\]
		which means  $\{n_{f_1},n_{f_2}\}=\{n_{f_3},n_{f_4}\}$.
		
		{\it Case III.} $f_1\neq f_3$ and $f_1f_2\neq f_3f_4$. 
		From $k_2\geq k_4$, \eqref{eq:prod-cong} and \eqref{eq-size-1}, we have
		$$
		\deg (G)\leq\deg (f_1f_2-f_3f_4)\leq \deg(f_1)+O(k_1)+\deg(f_2)\le  ck_1^2+ck_2^2+O(k_1).
		$$
		On the other hand, by the definition of $G$ and \eqref{eq:k2k4} we have
		$$
		\deg(G)=\sum_{i\le \alpha k_4}\deg(g_i)=\alpha^2k_4^2+O(k_4)\ge \alpha^4k_2^2-O(k_2).
		$$
		Therefore, 
		\begin{align*}
			\alpha^4k_2^2\le ck_1^2+ck_2^2+O(k_1),
		\end{align*}
		which implies that
		\begin{align}\label{eq-1234-1}
			\frac{k_2^2}{k_1^2}\le \frac{c}{\alpha^4-c}+O(1/k_1),
		\end{align}
		provided that $\alpha^4-c>0$. 
		From \eqref{eq:single-cong} and \eqref{eq-size-1} we have
		$$
		\deg (H)\le \deg(f_1-f_3)\le ck_1^2+O(k_1).
		$$
		By the definition of $H$ we see
		$$
		\deg (H)=\sum_{k_2+4<i\le L}\deg(g_i)\ge \alpha^2k_1^2-k_2^2-O(k_1).
		$$
		Therefore,
		\begin{align*}
			\alpha^2k_1^2-k_2^2\le ck_1^2+O(k_1),
		\end{align*}
		which implies that
		\begin{align}\label{eq-1234-2}
			\frac{k_2^2}{k_1^2}\ge \alpha^2-c-O(1/k_1).
		\end{align}
		Combining \eqref{eq-1234-1} with \eqref{eq-1234-2} gives
		$$
		\alpha^2-c-O(1/k_1)\leq\frac{k_2^2}{k_1^2}\leq \frac{c}{\alpha^4-c}+O(1/k_1).
		$$
		Since $K$ is sufficiently large, this is a contradiction to constraint 
		\[
		(\alpha^2-c)(\alpha^4-c)>c
		\]
		in \eqref{eq:parameters}. Thus, {\it Case III} can not occur. This completes the proof of our lemma.
	\end{proof}
    \begin{remark}
  For $\{n_{f_1},n_{f_2}\}=\{n_{f_3},n_{f_4}\}$, by Lemma \ref{lem:injective}, we obtain 
  $\{f_1,f_2\}=\{f_3,f_4\}$. This implies the same result as Pilatte's result \cite[Lemma 4.1]{Pi24}. 
\end{remark}
Moreover, we can estimate $\cS(x)$ by Lemma \ref{lem:injective}.
  \begin{lemma}\label{lem:sizeS}
For every sufficiently large $x$, we have
\[
\cS(x)=x^{c+o(1)}.
\]
\end{lemma}
\begin{proof}
For every sufficiently large $x$, let $k$ be the largest integer such that
\[
q^{k^2}\leq x.
\]
Then
\[
q^{k^2}\leq x<q^{(k+1)^2}.
\]
For any $n_f\in \cS\cap\{1,2,\ldots,\lfloor x\rfloor\}$,  by \eqref{ineq:1}, we obtain that $f\in \cF_i$ for some $i\leq k+1$. And by \eqref{ineq:2}, we obtain that every
$f\in\cF_{k-2}$ gives an element $n_f< q^{k^2}\leq x$. In addition, by
Lemma~\ref{lem:injective}, these elements $n_f$ are distinct. Therefore, 
\[
\sum_{\substack{c (k-2)^2\leq 2j < c (k-1)^2}} |\cP_{2j}|\leq \cS(x)\leq \sum_{i\leq k+1}\sum_{\substack{ci^2\leq 2j < c (i+1)^2}}|\cP_{2j}|.
\]
By Lemma \ref{lem:prime}, we obtain
\[
\cS(x)=q^{ck^2+O(k)}=x^{c+o(1)},
\]
which gives the result.
\end{proof}
	
	\section{\texorpdfstring{$\cS$}{S} is an asymptotic basis of order \texorpdfstring{$3-\eta$}{3-eta} with probability one}\label{Sec-5}

Let $m$ be a sufficiently large integer. We first estimate
		\[
		\mathbb P\big(n_{f_0}+n_{f_1}+n_{f_2}=m)
        \]
with $(f_0,f_1,f_2)$ in a suitable subset of $\cF^3$.

\subsection{Auxiliary lemmas} 
	\begin{lemma}\label{lem:expansion}
		Let $m\in \N$ be sufficiently large and $k$ be the largest integer such that
		\[
		10 \prod_{i\leq k}p(q^{2i-1}-1)\le m.
		\]
		Then $m=q^{k^2+O(k)}$. Let $\ell_*$ and $k_*$ be two integers such that $0\le \ell_* \le k_* \le k$. Then $m$ can be represented as
		\begin{equation}\label{eq:target-expansion}
			m=\overline{z y_k x_k \cdots y_1 x_1}^{\bb}
		\end{equation}
		with
		\[
		0\le x_i<q^{2i-1}-1,
		\qquad
		0\le y_i\le 3p/2,
		\qquad
		8\le z< 10q^{2k+2},
		\]
		satisfying the following properties:
		
		$(i)$ If $1\le i\le \ell_*$, then 
		$$
		\{y_i-2,y_i-1,y_i\}\subset \cA+\cA+\cA.
		$$
		
		$(ii)$ If $\ell_*<i\le k_*$, then 
		$$
		\{y_i-2,y_i-1,y_i\}\subset \cA+\cA+\cW.
		$$
		
		$(iii)$ If $k_*<i\le k$, then 
		$$
		\{y_i-1,y_i\}\subset \cW+\cW.
		$$
	\end{lemma}
	
	\begin{proof}
		By the definition of $k$, we have
		\begin{align}\label{eq-0622-3}
		10 \prod_{i\leq k}p(q^{2i-1}-1)\le m<10 \prod_{i\leq k+1}p(q^{2i-1}-1),
		\end{align}
		which clearly implies $m=q^{k^2+O(k)}$. It remains to show that $m$ has the desired expansion.
		
		Let $M_1=m$.
		Suppose that $M_{2i-1}$ has already been defined. Choose 
		\[
		x_i=\Big[M_{2i-1}\ {\rm mod}\ {q^{2i-1}-1}\Big]
		\]
		and
		\[
		M_{2i}=\frac{M_{2i-1}-x_i}{q^{2i-1}-1}.
		\]
		Then $M_{2i}$ is an integer, and 
		\[
		0\le x_i<q^{2i-1}-1.
		\]
		Next, we give definitions of $y_i$ and  $M_{2i+1}$ for any $1\le i\le k$.
		
		By Lemma~\ref{lem:aux}, the set $\cA+\cA+\cA$ contains $p+2$
		consecutive integers. Thus, there exists an integer $u$ with $0\leq u\leq p/2-2$ such that 
		\[
		\{u,u+1,\ldots,u+p+1\}\subset \cA+\cA+\cA. 
		\]
For $1\le i\le k_*$, we choose 
		\[
		y_i=
		\begin{cases}
			\big[M_{2i}\ {\rm mod}\  p\big], &\text{if~}\big[M_{2i}\ {\rm mod}\ p\big]\geq u+2,\\
			\big[M_{2i}\ {\rm mod}\ p\big]+p, &\text{otherwise}.
		\end{cases}
		\]
		So, for $1\le i\le k_*$, we have 
		$$
		\big\{y_i-2,y_i-1,y_i\big\}\subset \cA+\cA+\cA.
		$$
		Therefore, the property $(i)$ holds for any $0\leq \ell_{*}\leq k_{*}$.
		The property $(ii)$ follows immediately from the inclusion
		$$
		\cA+\cA+\cA\subset \cA+\cA+\cW.
		$$
		For $k_*<i\le k$, we choose the unique element $y_i$
		that satisfies
		\[
		y_i\equiv M_{2i}\bmod p \quad \text{and} \quad y_i\in\big\{1,2\ldots,p\big\}.
		\]
		Hence, $\big\{y_i-1,y_i\big\}\subset \cW+\cW$.
		Now, let
		\[
		M_{2i+1}=\frac{M_{2i}-y_i}{p}.
		\]
		By the above choices, we see that $M_{2i+1}$ is an integer and $0\le y_i\leq 3p/2$ for any $1\le i\le k$.
		
		After carrying out this for $i=1,\ldots,k$, we define
		\[
		z=M_{2k+1},
		\]
        and therefore 
        \[
        m=\overline{z y_k x_k \cdots y_1 x_1}^{\bb}.
        \]
		It remains to bound $z$. It is clear that
		\begin{align}\label{eq-0622-1}
			m\geq \overline{z \underbrace{00\cdots 00}_{2k\ \text{zeros}}}^{\bb}=z\prod_{j=1}^{2k}b_j=z\prod_{j=1}^{k}p\left(q^{2j-1}-1\right).
		\end{align}
		On the other hand,
		\begin{align}\label{eq-0622-2}
			m&\le  \overline{z \frac{3p}{2}\left(q^{2k-1}-2\right)\cdots \frac{3p}{2}\left(q-2\right)}^{\bb} \nonumber\\
			&\le \overline{z \underbrace{00\cdots 00}_{2k\ \text{zeros}}}^{\bb}+2\times\overline{(p-1)\left(q^{2k-1}-2\right)\cdots (p-1)\left(q-2\right)}^{\bb}\nonumber\\
			&\le \overline{(z+2) \underbrace{00\cdots 00}_{2k\ \text{zeros}}}^{\bb} \nonumber\\
			&=(z+2)p^k\prod_{j=1}^{k}(q^{2j-1}-1).
		\end{align}
		
		Combining \eqref{eq-0622-3}, \eqref{eq-0622-1} and \eqref{eq-0622-2}  gives
		$$
		10\prod_{i=1}^{k}p(q^{2i-1}-1)\leq (z+2)p^k\prod_{j=1}^{k}(q^{2j-1}-1)
		$$
		and
		$$
		z\prod_{j=1}^{k}p\left(q^{2j-1}-1\right)<10\prod_{i=1}^{k+1}p(q^{2i-1}-1),
		$$
	which implies that
		\[
		8\leq z< 10q^{2k+2}.
		\]
		Together with the previously established properties of the digits $x_i$ and $y_i$, this proves the lemma.
	\end{proof}

	Let $m$ be a sufficiently large integer. Choose $k$ as in Lemma \ref{lem:expansion}, and put
	\[
	k_*:=\lfloor\alpha k\rfloor,
	\qquad
	\ell_*:=\lfloor\alpha k_{*}\rfloor.
	\]
	Then
	\begin{equation}\label{eq:kl}
		k_*^2=\alpha^2 k^2+O(k)
		\quad \text{and} \quad
		\ell_*^2=\alpha^4k^2+O(k).
	\end{equation}
	Choose even integers $d$ and $d_0$ satisfying
	\begin{equation}\label{eq:d-degrees}
		ck^2\le d<c(k+1)^2,
		\qquad
		ck_{*}^2\le d_0<c(k_{*}+1)^2.
	\end{equation} 
	Let $x_i$ $(1\leq i\leq k)$ be defined in Lemma \ref{lem:expansion}. Recall that $\omega_i$ is one of the generators of $\big(\F_q[t]/(g_i)\big)^\times$. Define
	\begin{equation}\label{eq:cC}
			\cC_m=\big\{(f_0,f_1,f_2)\in \cP_{d_0}\times (\cP_d)^2:f_1\neq f_2,f_0f_1f_2\equiv \omega_i^{x_i}\bmod {g_i} \ \text{for} \ 1\leq i\leq \ell_{*}\big\}.
	\end{equation}
	
	For the probabilistic argument in Lemma \ref{lem:prob:equal}, we need the following set $\cC_m'\subset \cC_m$.
	
	\begin{lemma}\label{lem:set:norepeated}
		There is a subset $\cC_m^{\prime}\subset \cC_m$ such that no  polynomial $f$ can occur in the coordinate components of two distinct triples of $\cC_m^{\prime}$ and
		\[
		|\cC_m^{\prime}|= \frac{1}{2}|\cC_m|.
		\]
	\end{lemma}
	\begin{proof}
		Since $g_i\in\cP_{2i-1}$ and $d, d_0$ are even integers, we see that
		$$
		\big(f_0f_1f_2, g_i\big)=1
		$$
		for any $(f_0,f_1,f_2)\in \cC_m$.
		By the definition of $\cC_m$ in \eqref{eq:cC} and the Chinese Remainder Theorem, there exists a polynomial $P_{\ell_*}$ with
		$\big(P_{\ell_*},\prod_{i=1}^{\ell_{*}}g_i\big)=1$ such that 
		\[
		f_0f_1f_2\equiv P_{\ell_{*}}\Bigl(\operatorname{mod}\ \prod_{i=1}^{\ell_*} g_i\Bigr).
		\]
		
		Assume firstly that $f\in \cP_{d_0}$. Since $f$ has degree $d_0$, it can appear only in the first coordinate of the triple. Suppose that there are two triples $(f,f_1,f_2), (f,f_3,f_4)\in \cC_m$. Then
		\[
		f_1f_2\equiv f_3f_4\equiv P_{\ell_*}f^{-1}\Bigl(\operatorname{mod}\ \prod_{i=1}^{\ell_*} g_i\Bigr).
		\]
		Since $k$ is sufficiently large, we have
		\[
		\deg(f_1f_2-f_3f_4)\leq 2d=2ck^2+O(k)
		\]
		from \eqref{eq:d-degrees}.
		Moreover, from \eqref{eq:kl},
		$$
		\deg \left(\prod_{i=1}^{\ell_{*}}g_i\right)=\alpha^4k^2+O(k).
		$$
		Therefore, $f_1f_2=f_3f_4$ since $2c<\alpha^4$ by (\ref{eq:parameters}).  Note that $f_i$ $(1\leq i\leq 4)$ is irreducible. It follows that $\{f_1,f_2\}=\{f_3,f_4\}$ and the triples containing $f$ can only be $(f,f_1,f_2)$ or $(f,f_2,f_1)$.
		
		It remains to consider $f\in \cP_{d}$. Suppose that there are two triples $(f_5,f,f_6),$ $(f_7,f,f_8)\in \cC_m$. Then
		\[ 
		f_5f_6\equiv f_7f_8\equiv   P_{\ell_*}f^{-1}\Bigl(\operatorname{mod}\ \prod_{i=1}^{\ell_*} g_i\Bigr).
		\] 
		Similarly,
		\[
		\deg(f_5f_6-f_7f_8)\leq d+d_0\le 2d\le 2ck^2+O(k).
		\]
		This again forces $f_5f_6=f_7f_8$. Since $d_0\neq d$ and $f_i$ $(5\leq i\leq 8)$ is irreducible, we obtain $f_5=f_7, f_6=f_8$. In addition, if $(f_5,f,f_6)\in \cC_m$, then so does $(f_5,f_6,f)$. Therefore, the triples containing $f$ can only be $(f_5,f,f_6)$ and $(f_5,f_6,f)$. 
		
		The above discussion shows that 
		we can find a suitable subset $\cC_m^{\prime}$ with $|\cC_m^{\prime}|= \frac{1}{2}|\cC_m|$.
	\end{proof}
	We next estimate the size of $\cC_m^{\prime}$.
	\begin{lemma} \label{lem:sizeC}
		Let $\cC_m^{\prime}$ be defined above in Lemma \ref{lem:set:norepeated}. Then
		\[
		|\cC_m^{\prime}|=q^{\beta k^2 +O(k)}, 
		\]
		where $\beta=c(\alpha^2+2)-\alpha^4$.
	\end{lemma}
	\begin{proof}
		First, recall that each $g_i$ has odd degree and $d_0$, $d$ are even integers.
		Moreover, we have
		\[
		\deg \left(\prod_{i=1}^{\ell_{*}}g_i\right)={\ell_{*}}^2=\alpha^4k^2+O(k)
		\]
		and
		$$
		d_0+2d=c(\alpha^2+2)k^2+O(k).
		$$
		Let $P_{\ell_{*}}$ be defined as in Lemma \ref{lem:set:norepeated}.  From \eqref{choice}, $\prod_{i=1}^{\ell_{*}}g_i$ satisfies the hypotheses imposed on $g$ in Lemma \ref{lem:sawin} with $\theta=\frac{\ell_{*}^2}{d_0+2d}<0.9<1$. Thus, by applying Lemma \ref{lem:sawin} with
		\[
		g=\prod_{i=1}^{\ell_{*}}g_i\ \ \text{and} \ \ a=P_{\ell_{*}},
		\]
		we have
		\[
		|\cC_m|=T\left(P_{\ell_{*}};\prod_{i=1}^{\ell_{*}}g_i;d_0,d\right)=\frac{|\cP_{d_0}|\,|\cP_d|\big(|\cP_d|-1\big)}{\phi\left(\prod_{i=1}^{\ell_{*}}g_i\right)}
		+O\left(C_\theta^{\,d_0+2d}q^{\frac{d_0+2d-\ell_*^2}{2}}\right).
		\]
		Since $q$ is sufficiently large and
        \[
C_\theta^{\,d_0+2d}q^{\frac{d_0+2d-\ell_*^2}{2}}\ll q^{\frac{d_0+2d-\ell_*^2}{2}+\frac{3d\log C_{\theta}}{\log q}}\ll q^{((\alpha^2c+2c-\alpha^4)/2+3c\log C_{0.9}/\log q)k^2+O(k)}
        \]
    we get
        the error term is
		$
		\ll q^{\frac{3\beta k^2}{4}+O
			(k)}
		$
        from \eqref{choice}.
		By the definition of $\phi$, we have
		$$
		\phi\left(\prod_{i=1}^{\ell_{*}}g_i\right)=\prod_{i=1}^{\ell_{*}}(q^{\deg g_i}-1)=q^{\alpha^4k^2+O(k)}.
		$$
		Therefore, by Lemma \ref{lem:prime} we get
		$$
		\frac{|\cP_{d_0}|\,|\cP_d|\big(|\cP_d|-1\big)}{\phi\left(\prod_{i=1}^{\ell_{*}}g_i\right)}=q^{d_0+2d-\alpha^4k^2+O(k)}=q^{\beta k^2 +O(k)}.
		$$
		Our lemma now follows clearly from Lemma \ref{lem:set:norepeated}.
		\end{proof}
        
	The following is the key probabilistic estimate.	
			\begin{lemma}\label{lem:prob:equal}
			Let $m\in\N$ be sufficiently large. For any fixed $(f_0,f_1,f_2)\in \cC_m^{\prime}$, we have
			\[
			\mathbb P\big(n_{f_0}+n_{f_1}+n_{f_2}=m\big)
			\ge q^{-(1-\alpha^4)k^2-O(k)}.
			\]
		\end{lemma}
		
		\begin{proof}
			Write $m$ in the form in \eqref{eq:target-expansion}:
			$$
			m=\overline{zy_kx_k\cdots\ y_1x_1}^{\,\bb},
			$$
			where $k$, $x_i$, $y_i$ $(1\leq i\leq k)$ satisfy the properties asserted in Lemma \ref{lem:expansion}. We consider the probability $\mathbb P\big(n_{f_0}+n_{f_1}+n_{f_2}=m\big)$ with  $(f_0,f_1,f_2)\in\cC_m^{\prime}$. 
			For clarity, we shall use the following table: From top to bottom, each row represents the $\bb$ expansions of $n_{f_1}$, $n_{f_2}$, $n_{f_0}$, and $n_{f_0}+n_{f_1}+n_{f_2}$, respectively.
			
			\begin{table}[htbp]
				\centering
				\renewcommand{\arraystretch}{1.7}
				\setlength{\tabcolsep}{1.8mm}
				\adjustbox{max width=\textwidth}{%
					\begin{tabular}{c|*{13}{c}}
						\hline
						$n_{f_1}$
						&
						$s(f_1)$
						&
						$R_k(f_1)$
						&
						$E_k(f_1)$
						&
						$\cdots$
						&
						$E_{k_{*}+1}(f_1)$
						&
						$r_{k_{*}}(f_1)$
						&
						$e_{k_{*}}(f_1)$
						&
						$\cdots$
						&
						$r_{\ell_{*}+1}(f_1)$
						&
						$e_{\ell_{*}+1}(f_1)$
						&
						$r_{\ell_{*}}(f_1)$
						&
						$e_{\ell_{*}}(f_1)$
						&
						$\cdots$
						\\
						
						$n_{f_2}$
						&
						$s(f_2)$
						&
						$R_k(f_2)$
						&
						$E_k(f_2)$
						&
						$\cdots$
						&
						$E_{k_{*}+1}(f_2)$
						&
						$r_{k_{*}}(f_2)$
						&
						$e_{k_{*}}(f_2)$
						&
						$\cdots$
						&
						$r_{\ell_{*}+1}(f_2)$
						&
						$e_{\ell_{*}+1}(f_2)$
						&
						$r_{\ell_{*}}(f_2)$
						&
						$e_{\ell_{*}}(f_2)$
						&
						$\cdots$
						\\
						
						$n_{f_0}$
						&
						&
						&
						&
						&
						$s(f_0)$
						&
						$R_{k_{*}}(f_0)$
						&
						$E_{k_{*}}(f_0)$
						&
						$\cdots$
						&
						$R_{\ell_{*}+1}(f_0)$
						&
						$E_{\ell_{*}+1}(f_0)$
						&
						$r_{\ell_{*}}(f_0)$
						&
						$e_{\ell_{*}}(f_0)$
						&
						$\cdots$
						\\
						\hline
						
						$\sum_{i=0}^2n_{f_i}$
						&
						$z'$
						&
						$y_k'$
						&
						$x_k'$
						&
						$\cdots$
						&
						$x_{k_{*}+1}'$
						&
						$y_{k_{*}}'$
						&
						$x_{k_{*}}'$
						&
						$\cdots$
						&
						$y_{\ell_{*}+1}'$
						&
						$x_{\ell_{*}+1}'$
						&
						$y_{\ell_{*}}'$
						&
						$x_{\ell_{*}}'$
						&
						$\cdots$
						\\
						\hline
					\end{tabular}%
				}
			\end{table}
			Hence,
			\[
			n_{f_0}+n_{f_1}+n_{f_2}=\overline{z^{\prime}y_k^{\prime}x_k^{\prime}\cdots y_1^{\prime}x_1^{\prime}}^{\,\bb},
			\]
			where $x_i', y_i'$ and $z'$ are defined as follows:
			
			\begin{itemize}
				\item If $1\leq i\le {\ell_{*}}$, then
				\[
				\begin{aligned}
					x_i'
					&=
					\left[\sum_{j=0}^{2} e_i(f_j)
					\mod q^{2i-1}-1\right],\\
					y_i'
					&=
					\sum_{j=0}^{2}r_i(f_j)+\kappa_i,\quad \kappa_i=\left\lfloor
					\frac{e_i(f_1)+e_i(f_2)+e_i(f_0)}
					{q^{2i-1}-1}
					\right\rfloor
					\in\{0,1,2\}.
				\end{aligned}
				\]
				
				\item If ${\ell_{*}}< i\le k_{*}$, then
				\[
				\begin{aligned}
					x_i'
					&=
					\left[\sum_{j=1}^{2} e_i(f_j)+E_i(f_0)
					\mod q^{2i-1}-1\right],\\
					y_i'
					&=
					\sum_{j=1}^{2}r_i(f_j)+R_i(f_0)+\kappa_i,\quad \kappa_i=\left\lfloor
					\frac{e_i(f_1)+e_i(f_2)+E_i(f_0)}
					{q^{2i-1}-1}
					\right\rfloor
					\in\{0,1,2\}.
				\end{aligned}
				\]
				
				\item If $i=k_{*}+1$, then
				\[
				\begin{aligned}
					x_i'
					&=
					\left[\sum_{j=1}^{2}E_i(f_j)+s(f_0)
					\mod q^{2i-1}-1\right],\\
					y_i'
					&=
					\sum_{j=1}^{2}R_i(f_j)+\kappa_i,\quad \kappa_i=\left\lfloor
					\frac{E_i(f_1)+E_i(f_2)+s(f_0)}
					{q^{2i-1}-1}
					\right\rfloor.
				\end{aligned}
				\]
				
				\item If $k_{*}+1<i\le k$, then
				\[
				\begin{aligned}
					x_i'
					&=
					\left[\sum_{j=1}^{2}E_i(f_j)
					\mod q^{2i-1}-1\right],\\
					y_i'
					&=
					\sum_{j=1}^{2}R_i(f_j)+\kappa_i,\quad \kappa_i=\left\lfloor
					\frac{E_i(f_1)+E_i(f_2)}
					{q^{2i-1}-1}
					\right\rfloor\in \{0,1\}.
				\end{aligned}
				\]
				
				\item Finally,
				$z^{\prime}=s(f_1)+s(f_2).$
			\end{itemize}
			
			What we need to consider first are $y_{k_{*}+1}^{\prime}$. To ensure that $\kappa_{k_{*}+1}\in\{0,1\}$,
			we impose
			$$
			s(f_0)=1.
			$$
			Let $Q_i=q^{2i-1}-1$ for $1\leq i\leq k$. By the definition of $\cC_m$, we obtain
			\[
			\sum_{j=0}^{2} e_i(f_j)\equiv x_i \bmod {Q_i}
			\]
			for $1\leq i\leq \ell_{*}$, which implies that $x_i^{\prime}=x_i$  since $0\leq x_i<Q_i$. Therefore,
			\begin{align}\label{p}
				\mathbb P(n_{f_0}+n_{f_1}+n_{f_2}=m)
				&\geq  \mathbb P(s(f_0)=1,z^{\prime}=z,y_i^{\prime}=y_i, x_i^{\prime}=x_i \   \text{for}\ 1\leq i\leq k )\notag\\
				&=\mathbb P(z^{\prime}=z)\prod_{i=1}^{\ell_{*}}\mathbb P(y_i^{\prime}=y_i)\prod_{i=\ell_{*}+1,i\neq k_{*}+1}^{k}\mathbb P( x_i^{\prime}=x_i,y_i^{\prime}=y_i) \\
				&\ \ \ \ \times \mathbb P(s(f_0)=1,x_{k_{*}+1}^{\prime}=x_{k_{*}+1},y_{k_{*}+1}^{\prime}=y_{k_{*}+1}), \notag
			\end{align}
			where the last equality follows from the mutual independence of the random variables.
			
			For $1\le i\le \ell_{*}$, since $\{y_i-2,y_i-1,y_i\}\subset \cA+\cA+\cA$ by Lemma \ref{lem:expansion}, there exists one choice for $r_i(f_j)\in\cA$ $(0\le j\leq 2)$ such that
			$$
			y_i'=\sum_{j=0}^{2}r_i(f_j)+\kappa_i=y_i.
			$$
			Thus, 
			\begin{align}\label{p1}
				\mathbb P(y_i^{\prime}=y_i)\geq \frac{1}{p^3}, \ \ \text{for} \ \ 1\le i\le \ell_{*}.
			\end{align}
			
			For $\ell_{*}<i\leq k_{*}$, even though $e_i(f_1)$ and $e_i(f_2)$  are fixed, since  $E_i(f_0)$ is a random variable uniformly distributed on 
			\[ 
			\{0,1,\ldots,q^{2i-1}-2\}, 
			\] 
			as $E_i(f_0)$ varies, 
			$$
			\sum_{j=1}^{2} e_i(f_j)+E_i(f_0)
			$$
			runs through all residue classes modulo $Q_i$, which implies that
			there is a unique admissible value of  $E_i(f_0)$ such that
			$
			x_i^{\prime}= x_i
			$ for each $i$. Moreover, since $\{y_i-2,y_i-1,y_i\}\subset \cA+\cA+\cW$ by Lemma \ref{lem:expansion}, there exists one choice for $(r_i(f_1),r_i(f_2),R_i(f_0))\in \cA\times \cA\times\cW$ such that
			$$
			y_i'=
			\sum_{j=1}^{2}r_i(f_j)+R_i(f_0)+\kappa_i=y_i,
			$$
			no matter what $\kappa_i\in\{0,1,2\}$ is.  Thus, 
			\begin{align}\label{p2}
				\mathbb P( x_i^{\prime}=x_i,y_i^{\prime}=y_i)\geq \frac{1}{Q_ip^3} \ \ \text{for} \ \ \ell_{*}<i\leq k_{*}.
			\end{align}

			For $k_{*}<i\le k$, since $E_i(f_j)$ $(j=1,2)$ is a random variable uniformly distributed on 
			\[ 
			\{0,1,\ldots,q^{2i-1}-2\}. 
			\]  there exists one choice for $E_i(f_1)$ if we fix $E_i(f_2)$ such that $x_{i}^{\prime}=x_i$.  Moreover, $\{y_i-1,y_i\}\subset W+W$ by Lemma \ref{lem:expansion}, there exists one choice for $(R_i(f_1),R_i(f_2))\in \cW\times\cW$ such that
			$$
			y_i'=
			\sum_{j=1}^{2}R_i(f_j)+\kappa_i=y_i,
			$$
			no matter what $\kappa_i\in\{0,1\}$ is.
			Thus, 
			\begin{align}\label{p3}
				\mathbb P(s(f_0)=1,x_{k_{*}+1}^{\prime}=x_{k_{*}+1},y_{k_{*}+1}^{\prime}=y_{k_{*}+1})\geq \frac{1}{q^{3k_{*}}Q_{k_{*}+1}p^2},
			\end{align}
			and
			\begin{align}\label{p4}
				\mathbb P( x_i^{\prime}=x_i,y_i^{\prime}=y_i)\geq \frac{1}{Q_ip^2} \ \ \text{for} \ \ k_{*}+1<i\leq k.
			\end{align}

			Finally, since $s(f_1)+s(f_2)\in \{2,3\cdots, 2q^{3k}\}$ and $8\le z\le 10q^{2k+2}<q^{3k}$, we obtain
			\begin{align}\label{p5}
				\mathbb P(z^{\prime}=z)\geq1/q^{6k}.
			\end{align}
			Combining \eqref{p}, \eqref{p1}, \eqref{p2}, \eqref{p3}, \eqref{p4} and \eqref{p5} together gives
			\begin{align*}
				\mathbb{P}\big(n_{f_0}+n_{f_1}+n_{f_2}=m\big)&\geq \frac{1}{q^{6k}}\frac{1}{p^{3\ell_{*}}}\frac{1}{q^{3k_{*}}}\prod_{i=l_{*}+1}^{k}\frac{1}{Q_ip^3}=\frac{1}{q^{(1-\alpha^4)k^2+O(k)}},
			\end{align*}
			which implies the result.
		\end{proof}

	\subsection{Proof of Theorem \ref{thm:main}}
   By Lemmas~\ref{prop:sidon} and \ref{lem:sizeS}, the random set $\cS$ is indeed an infinite Sidon set. It therefore remains to prove that for any $0<\eta<0.0527$, $\cS$ is an asymptotic basis of order $3-\eta$ with positive probability. We prove a stronger result that $\cS$ is an asymptotic basis of order $3-\eta$ with probability one by using the following Borel--Cantelli lemma; see \cite[p.135]{HR83}.
	\begin{lemma}[Borel--Cantelli Lemma] 
		Let $\cE_1,\cE_2,\ldots$ be a sequence of events in a probability space. If 
		\[ 
		\sum_{m=1}^{\infty}\mathbb P(\cE_m)<\infty, \] 
		then 
		\[ 
		\mathbb P\Big(\bigcap_{M=1}^{\infty}\bigcup_{m=M}^{\infty}\cE_m\Big)=0. 
		\] 
		Equivalently, with probability one, only finitely many of the events $\cE_m$ occur. 
	\end{lemma} 
	
		In our application, for each positive integer $m$, let $\cE_m$ be the event that there does not exist $s_1,s_2,s_3\in \cS$ such that $m=s_1+s_2+s_3$ with some $s_i\leq m^{1-\eta}$.
	
	For all sufficiently large $m$ and any $(f_0,f_1,f_2)\in \cC_m^{\prime}$, since $m=q^{k^2+O(k)}$, Lemma \ref{lem:size} and \eqref{eq:kl} give
\[
n_{f_0}=q^{k_{*}^2+O(k_{*})}=q^{\alpha^2 k^2+O(k)}=m^{\alpha^2+o(1)}.
\]
Therefore, since $\alpha^2<1-\eta$,  we obtain
\begin{align}\label{ineq:nf0}
n_{f_0}\le m^{1-\eta}.
\end{align}
In addition, the triples in $\cC_m^{\prime}$ do not share any polynomial. Hence the corresponding events
\[
\{n_{f_0}+n_{f_1}+n_{f_2}=m\},
\qquad (f_0,f_1,f_2)\in \cC_m^{\prime},
\]
are mutually independent. Therefore, by the definition of $\cE_m$, \eqref{ineq:nf0}, Lemmas \ref{lem:sizeC} and \ref{lem:prob:equal}, we obtain
\begin{align*}
	\mathbb P(\cE_m)
    &= \mathbb P\Big( \bigcap_{(f_0,f_1,f_2)\in \cF^3} \Big\{ m\neq n_{f_0}+n_{f_1}+n_{f_2}\ \text{or} \ \min_{0\le i\le 2} n_{f_i}> m^{1-\eta}\Big\} \Big) \\ 
    &\le \mathbb P\Big( \bigcap_{(f_0,f_1,f_2)\in \cC_m^{\prime}}\{ m\neq n_{f_0}+n_{f_1}+n_{f_2}\} \Big)\\
	&=\prod_{(f_0,f_1,f_2)\in \cC_m^{\prime}}(1-\mathbb P(\{ m= n_{f_0}+n_{f_1}+n_{f_2}\} ))\\
	&\le (1-q^{-(1-\alpha^4)k^2-O(k)})^{q^{\beta k^2+O(k)}}.
\end{align*}
Using the inequality $1-x\leq e^{-x}$,
\[
\mathbb P(\cE_m)
\le
e^{-q^{(\beta-(1-\alpha^4))k^2-O(k)}}.
\]
By \eqref{eq:beta-gamma},
\[
\beta-(1-\alpha^4)=c(2+\alpha^2)-1=\gamma>0.
\]
For all sufficiently large $m$, since $m=q^{k^2+O(k)}$, the failure probability is at most
$
e^{-m^{\gamma-o(1)}},
$
which implies that 
$$
\sum_{m=1}^{\infty}\mathbb P(\cE_m)<\infty.
$$
Therefore, by the Borel--Cantelli lemma, with probability 1, $\cS$ is an asymptotic basis of order $3-\eta$.

Since $\alpha^2=0.9473$ in \eqref{choice} and $1-\alpha^2=0.0527$, we obtain the desired conclusion for every
$0<\eta<0.0527$.

\section{Remarks and open problems}\label{sec:conclusion}
As shown by Pilatte \cite[Lemma 4.1]{Pi24}, the parameter $c$ satisfies
\[
1-c>\frac{c}{1-c},
\]
which implies that 
\[
0<c<\frac{3-\sqrt{5}}{2}.
\]
In our argument, the parameter $c$ in our construction is subject to the same upper bound since $c<(\alpha^2-c)(\alpha^4-c)<(1-c)^2$. And the range of $\eta$ is determined by the range of $\alpha^2$. Since $c(2+\alpha^2)>1$, the restriction on $c$ forces  $\eta$ to be small. Thus, extending the admissible range of $c$ while preserving the
asymptotic basis property would lead to a larger admissible range of
$\eta$. One possible approach is to use the deletion method in
Cilleruelo's construction \cite{Ci14} to improve the range of $c$.
However, such a deletion may remove some triples from $\cC_m'$, and it
seems difficult to obtain a sufficiently precise estimate for the
number of admissible triples that remain.			
            
Motivated by this $3-\eta$ version of the asymptotic Sidon basis problem, Csaba S\'andor asked the following more general problem:
\begin{problem}[$a+b+1$ version] 
Let $0\leq a\leq b\leq 1$ be real numbers satisfying  
$$
a+b>1. 
$$
Does there exist a Sidon set $\cS\subset\mathbb N$ such that every sufficiently large integer $m$ has a representation 
$$
m=s_1+s_2+s_3,\qquad s_1,s_2,s_3\in \cS, 
$$
with
$$ s_1\ll m^a,\qquad s_2\ll m^b,\qquad s_3\asymp m? 
$$
\end{problem}

The range $a+b>1$ is the natural necessary condition suggested by the
classical upper bound $\cS(x)\ll x^{1/2}$ for any infinite Sidon set $\cS$. 
And the
range $a+b>\sqrt{2}$ may be a more realistic first target from Ruzsa's exponent $\sqrt{2}-1$ for dense infinite Sidon sequences \cite{Ru98}. In addition, applying our truncated construction should give a result for a small range of $a$ and $b$. 

We also propose the following problem, in the spirit of
additive problems with almost equal summands.
\begin{problem}[Almost equal summands version]
Does there exist a constant $0<\theta<1$ and an infinite Sidon set
$\cS\subset\mathbb N$ such that every sufficiently large integer $m$
can be written as
\[
    m=s_1+s_2+s_3,
    \qquad s_1,s_2,s_3\in\cS,
\]
with
\[
    \left|s_i-m/3\right|
    \ll m^\theta
    \qquad (1\leq i\leq 3)?
\]
More generally, determine the smallest admissible value of $\theta$.
\end{problem}

Unfortunately, our present construction cannot force all three summands to lie in short intervals around $m/3$.
			
We also mention related works on $B_h[g]$-sets, also known as
generalized Sidon sets. A set $\cA\subset\N$ is called a $B_h[g]$-set if every positive integer has at most $g$ different representations of the form
\[
n=a_1+\cdots+a_h,
\ \
a_1\leq\cdots\leq a_h,
\ \
a_1,\ldots,a_h\in\cA.
\]
In particular, $B_2[1]$-sets are precisely Sidon sets, and one may ask similar questions to those on Sidon sets as above.
Kiss and S\'andor \cite{KS21} proved that for every $h\geq 2$, there exists a
$B_h[1]$-set which is an asymptotic basis of order $2h+1$. They later
improved the order from $2h+1$ to $2h$ in \cite{KS25}, and also asked the following problems:
\begin{problem}
Determine the smallest value of $k = k(h)\geq h$ for which there exists a
$B_h[1]$-set which is an asymptotic basis of order $k$.
\end{problem}

\begin{problem}
Determine the smallest value of $g = g(h)$ for which there exists an asymptotic
basis of order $h + 1$ which is a $B_h[g]$-set.
\end{problem}

Pilatte \cite{Pi24} mentioned that a simple modification of his proof could guarantee $k(h)\leq 2h-1$. It should be challenging to know if this could be improved for larger $h\geq 3$, in view of the fact that the asymptotic expected (maximum) size of a $B_h[1]$-set in $[1,N]$ is $N^{1/h}.$

\section{Acknowledgments} 
The author would like to express his sincere gratitude to his advisors Professor Ping Xi and Professor P\'eter P\'al Pach for their support and guidance.
The author would also like to thank Professor Ping Xi for suggesting the $3-\eta$ problem and for helpful discussions, thank Yuchen Ding and Xiamiao Zhao for suggestions and discussions and thank Csaba S\'andor for suggesting the $a+b+1$ problem. 
This work is supported in part by Shaanxi NSF (No. 2025JC-QYCX-002), Shaanxi Fundamental Science Research Project for Mathematics and Physics (No.25JSZ007) and China Scholarship Council.

	\bibliographystyle{plain}

\end{document}